\documentclass{amsart}

\usepackage{a4wide,amsmath,amssymb}
\usepackage[toc,page]{appendix}

\numberwithin{equation}{section}

\usepackage{color,enumerate}

\newtheorem{thm}{Theorem}[section]
\newtheorem{prop}[thm]{Proposition}
\newtheorem{lem}[thm]{Lemma}
\newtheorem{cor}[thm]{Corollary}
\newtheorem{rem}{Remark}[section]
\newtheorem{de}{Definition}[section]
\newtheorem{ex}{Example}[section]

\def\eps{\varepsilon}
\def\a{\alpha}
\def\b{\beta}
\def\de{\partial}
\def\d{\delta}
\def\g{\gamma}

\def\s{\sigma}

\def\de{\partial}

\newcommand{\ol}{\overline}

\newcommand{\Z}{{\mathbb Z}}

\newcommand{\R}{{\mathbb R}}

\newcommand{\Om}{\Omega}

\begin{document}
\date{\today}

\title[]{
Uniqueness of solutions in Mean Field Games\\with 
several populations and Neumann conditions
}
\author[
]{Martino Bardi 
 {and} Marco Cirant
 }
 \address{Department of Mathematics "T. Levi-Civita", University of Padova, Via Trieste 63, 35121 Padova, Italy} \email{bardi@math.unipd.it, cirant@math.unipd.it}
 \thanks{
 The authors are members of the Gruppo Nazionale per l'Analisi Matematica, la Probabilit\`a e le loro Applicazioni (GNAMPA) of the Istituto Nazionale di Alta Matematica (INdAM). They are partially supported by the research projects  "Mean-Field Games and Nonlinear PDEs" of the University of Padova and ``Nonlinear Partial Differential Equations: Asymptotic Problems and Mean-Field Games" of the Fondazione CaRiPaRo.}
%
\begin{abstract} We  study the uniqueness of solutions to systems of PDEs arising in Mean Field Games with several populations of agents and Neumann boundary conditions. The main assumption  requires the smallness of some data, e.g., the length of the time horizon. This complements the existence results for MFG models of segregation phenomena introduced by the authors and Achdou. An application to  robust Mean Field Games is also given.
\end{abstract}
\maketitle
\tableofcontents
\section{Introduction}

The systems of partial differential equations associated to finite-horizon Mean Field Games (briefly, MFGs) with $N$ populations of agents have the form
\begin{equation}
\label{mfg-syst-0}
\left\{
\begin{array}{ll}
- \partial_t v_k - 
 \Delta v_k + H_k(x,Dv_k)  = F_k(x, m(t,\cdot)) , & \textit{in }  (0,T)\times\Omega , \\ \\
\partial_t m_k - 
 \Delta m_k -  \text{div}(D_p H_k(x,Dv_k)m_k) = 0 & \textit{in } (0,T)\times\Omega, \\ \\
v_k(T,x) = G_k(x, m(T,\cdot)) , \; m_k(0,x) = m_{0,k}(x) & \textit{in } \Omega , \quad k=1,\dots,N ,
\end{array}
\right.
\end{equation}
where the unknown $m$ is a vector of probability densities on $\Omega$, $F_k$ and $G_k$ are function of this vector and represent the running and terminal costs of a representative agent of the $k$-population, and $v_k$ is the value function of this agent. The first $N$ equations are parabolic of Hamilton-Jacobi-Bellman type and backward in time with a terminal condition, the second $N$ equations are parabolic of Kolmogorov-Fokker-Planck  type and forward in time with an initial condition. If the state space $\Omega \subseteq \R^d$ is not all $\R^d$, boundary conditions must also be imposed. In most of the theory of MFGs they are periodic, which are the easiest to handle, here we will consider instead Neumann conditions, i.e., 
\begin{equation}
\label{Neumann}
\partial_n v_k = 0, \quad 
\partial_n m_k + m_k D_p H_k(x, Dv_k) \cdot n = 0  
\quad \text{on }  (0,T)\times\partial \Omega .
\end{equation}
There is a large literature on the existence of solutions for these equations, especially in the case of a single population $N=1$, beginning with the pioneering papers of Lasry and Lions \cite{LLcras1, LLcras, LL}  and Huang, Caines and Malhame \cite{HCM:06, HCM:07ieee, HCM:07jssc}
, see the lecture notes \cite{Car, GLL}, the books \cite{GomesBook, CaDe}, the survey \cite{GS}, and the references therein.  
 Systems with several populations, $N>1$, were treated with Neumann conditions in \cite{Cir, CirVe} for the stationary case and in \cite{ABC} in the evolutive case, with periodic conditions in \cite{BFel, CarPoTo}. 

Uniqueness of solutions is a much more delicate issue. For one population  Lasry and Lions \cite{LLcras1, LLcras, LL}  discovered a monotonicity condition on the costs $F$ and $G$ that together with the 
 convexity in $p$ of the Hamiltonian  $H(x,p)$ implies the uniqueness of classical solutions. It reads
   \begin{equation}
   \label{monoF}
  \int_\R (F(x,\mu) - F(x,\nu)) d(\mu -\nu)(x) ) > 0 , \quad \text{if } \; 
  \mu\ne 
  \nu 
  \end{equation}
and it means that a representative agent prefers the regions of the state space that are less crowded. This is a restrictive condition that is satisfied in some models and not in others. 
When it fails, non-uniqueness may arise: this was first observed in the stationary case by Lasry and Lions \cite{LL} and other counterexamples were shown by \cite{Gu, B, BP, GNPraz}. The need of a condition such as \eqref{monoF} for having uniqueness for finite-horizon MFGs 
 was discussed at length in \cite{L}, and some explicit examples of non-uniqueness appeared very recently in \cite{BriCar}, \cite{CT}, and in  \cite{BF}  that presents also a probabilistic proof and references on other examples obtained by the probabilistic approach.
 
 For multi-population problems, $N>1$, there are extensions of the monotonicity condition \eqref{monoF} in \cite{Cir, BF} and they are even more restrictive: they impose not only aversion to crowd within each population, but also that the costs due to this effect dominate the costs due to the interactions with the other populations. This is not the case in the multi-population models of segregation in urban settlements proposed in \cite{ABC} following the ideas of the Nobel Prize Thomas Schelling \cite{Sch}. There the interactions between two different populations are the main cause of the dynamics, and in fact examples of multiple solutions were shown in \cite{ABC} and \cite{CirVe} for the stationary case and in \cite{BF} for the evolutive one. Therefore a different criterion giving uniqueness in some cases is particularly desirable when $N>1$.

A second regime for uniqueness was introduced in a lecture of P.L. Lions on January 9th, 2009 
\cite{L}: it occurs if the length $T$ of the time horizon is short enough. To our knowledge Lions' original argument did not appear 
 in print. For finite state MFGs, uniqueness for short time was proved by Gomes, Mohr, and Souza \cite{GMS} as part of their study of the large population limit. 
 For continuous state, an existence and uniqueness result under a ''small data'' condition was given 
  in \cite{HCM:07jssc} for Linear-Quadratic-Gaussian MFGs using a contraction mapping argument to solve the associated system of Riccati differential equations, and similar arguments were used for different classes of linear-quadratic problems 
  in \cite{wangzhang12, moonbasar16}. 
The well-posedness when $H(x, Dv) -F(x,m)$ is replaced by $\eps \mathcal{H}(x, Dv, m)$ with $\eps$ small is studied in \cite{Amb}, and another result for small Hamiltonian 
  is in \cite{T} for nonconvex $H$.

Very recently the first author and Fischer  \cite{BF}  revived Lions' argument to show that the smoothness of the Hamiltonian is the crucial property to have small-time uniqueness without monotonicity of the costs and convexity of $H$, and gave an example of non-uniqueness for all $T>0$ and $H(x,p)=|p|$. The uniqueness theorem for small data in \cite{BF} holds for $N=1$ and $\Omega=\R^d$ with conditions on the behaviour of the solutions at infinity. 

In the present paper we focus instead on $N\geq 1$ and Neumann boundary conditions, which is  the setting of the MFG models of segregation in \cite{ABC}. The 
new difficulties arise 
 from the boundary conditions, that require different methods for some estimates, especially on the $L^\infty$ norm of the densities $m_k$. Our first uniqueness result assumes a suitable smoothness of the Hamiltonians $H_k$, but neither convexity nor growth conditions, and that the costs $F_k, G_k$ are Lipschitz in $L^2$ with respect to the measure $m$, with no monotonicity.
The smallness condition on the data depends on the range of the spacial gradient of the solutions $v_k$, unless $D_pH_k$ are bounded and globally Lipschitz for all $k$. Then we complement such result with some a priori gradient estimates on $v_k$, under an additional  quadratic growth condition on $H_k$ and some more regularity of the costs, 
and get a $\bar T>0$ depending only on the data such that there is uniqueness for all horizons $T\leq\bar T$. Finally, we 
 give sufficient conditions ensuring both existence and uniqueness for the system \eqref{mfg-syst-0} with the boundary conditions \eqref{Neumann}, as well as for some robust MFGs considered in \cite{BBT, moonbasar16}, which are interesting examples with nonconvex Hamiltonian.

We mention that in the stationary case, uniqueness up to (space) translation may hold without \eqref{monoF} in force. A special class of MFG on $\R^d$ enjoying such a feature has been identified in \cite{Cir16}.

The paper is organised as follows. Section \ref{uniq} contains the main result about uniqueness for small data, possibly depending on gradient bounds on the solutions. Section \ref{appl} gives further sufficient conditions depending only on the data for uniqueness and existence of solutions. The Appendix \ref{app} recalls  a comparison principle for HJB equations with Neumann conditions. 

\section{The uniqueness theorem}\label{uniq}
Consider the MFG system for $N$ populations
\begin{equation}\label{mfg-syst}
\left\{
\begin{array}{ll}
- \partial_t v_k - 
 \Delta v_k + H_k(x,Dv_k)  = F_k(x, m(t,\cdot)) , & \textit{in }  (0,T)\times\Omega , \\ \\
\partial_t m_k - 
 \Delta m_k -  \text{div}(D_p H_k(x,Dv_k)m_k) = 0 & \textit{in } (0,T)\times\Omega, \\ \\
\partial_n v_k = 0, \; 
\partial_n m_k + m_k D_p H_k(x, Dv_k) \cdot n = 0  & \textit{on }  (0,T)\times\partial \Omega , \\ \\
v_k(T,x) = G_k(x, m(T,\cdot)) , \; m_k(0,x) = m_{0,k}(x) & \textit{in } \Omega
\end{array}
\right.
\end{equation}
where $k= 1,\dots, N$, $Dv_k
$ denotes the gradient of the $k-th$ component $v_k$ of  the unknown $v$ with respect to the space variables, $\Delta$ is the Laplacian with respect to the space variables $x$,  $D_pH_k
$ is the gradient of the Hamiltonian of the $k$-th population with respect to the moment variable, $\Omega\subseteq\R^d$ is a bounded open set with 
 boundary $\de \Omega$ of class $C^{2,\beta}$ for dome $\beta>0$, and $n(x)$ is its exterior normal at $x$.  
The components $m_k$ of the unknown vector $m$ are bounded densities of probability measures on $\Omega$, i.e., 
 $m$ lives in
\[
{ \mathcal{P}_N}(\Omega
) := \left\{\mu=(\mu_1,\dots,\mu_N)\in L^\infty(\Omega
)^N : \mu_k\geq 0,\; \int_\Omega \mu_k(x) dx = 1 
\right\}.
\]
$F$ and $G$ represent, respectively, the running and terminal cost of the MFG 
 $$
 F :  \overline\Omega\times \mathcal{P}_N(\Omega)\to \R^N , \quad G :  \overline\Omega\times \mathcal{P}_N(\Omega)\to \R^N.
 $$
By classical solutions we will mean functions of $(t,x)$ of class $C^1$ in $t$ and $C^2$ in $x$ in $[0,T]\times \ol \Omega$.

\subsection{The main result} 

Our main assumptions are the smoothness of the Hamiltonians and a 
Lipschitz continuity of the costs in the norm $\|\cdot\|_2$ of $L^2(\Omega
)^N$ that we state next.
We consider $H_k : \ol\Om\times\R^d \to \R$ continuous and satisfying
\begin{equation}
\label{ass:H}
 D_pH_k(x,p) \text{ is continuous and locally Lipschitz in $p$ uniformly in }
 x\in\ol\Om .
\end{equation}

We will assume $F, G$ 
 satisfy, for all $\mu, \nu$, 
\begin{equation}
\label{Flip}
\|F(\cdot,\mu)-F(\cdot,\nu)\|_2^2 \leq L_F \|\mu-\nu\|_2^2 , 
\end{equation}
\begin{equation}\label{Glip}
\|DG(\cdot,\mu)-DG(\cdot,\nu)\|_2^2 \leq L_{
G} \|\mu-\nu\|_2^2 ,
\end{equation}

%
\begin{thm}
\label{st-uniq} 
Assume \eqref{ass:H}, 
 \eqref{Flip}, \eqref{Glip}, $m_0\in { \mathcal{P}_N}
(\Omega)$, and $(\tilde v, \tilde m), (\ol v, \ol m)$ are two classical solutions of \eqref{mfg-syst}. 
Denote
\[
\mathcal{C}:= \text{\rm co}\{ D\tilde v(t,x), D\ol v(t,x) : (t,x)\in (0,T)\times\Omega \},
\]
\begin{equation}
\label{ass:Hv1} 
C_H:= \max_{k= 1,\dots, N} \sup_{
x\in\Omega,\, p\in \mathcal{C}} |D_pH_k(x,p)| 
 , 
 \end{equation}
 \begin{equation}
\label{ass:Hv2} \bar C_H := \max_{k= 1,\dots, N} \sup_{x\in\Omega,\, p, q\in \mathcal{C}} \frac{|D_pH_k(x,p)-D_pH_k(x,q)|}{|p-q|}  
\end{equation}
%
%

\noindent
Then there exists a function $\Psi$ of $T, L_F, L_G, C_H, \bar C_H, N
$ and $\max_k\| m_{0,k}\|_\infty$ (depending also on $\Omega$), such that the inequality $\Psi<1$ implies $\tilde v(t,\cdot) = \ol v(t,\cdot)$ and $\tilde m(t,\cdot) = \ol m(t,\cdot)$ for all $t\in[0,T]$, 
 and $\Psi<1$ holds if either $T$, or $\bar C_H$, or the pair $L_F, L_G$ is small enough.

 \end{thm}
 %
For the the proof we need two auxiliary results.
 
  \begin{prop}\label{prop:est}
  \noindent 
 There are constants $r > 1$ and $C > 0$ depending only on $d$ and $\Omega$ such that
 \begin{equation}
 \label{est_m}
 \|m_k\|_{L^{\infty}((0,T)\times\Omega)} \leq C [1+ \| m_{0,k}\|_\infty + (1+T)\|D_p H_k(\cdot,Dv_k) \|_{L^\infty((0,T)\times \Omega)}]^r,  \quad k=1, \ldots, N.
 \end{equation}
 \end{prop}
 \begin{proof} {\em Step 1.} We aim at proving that for any $q \in [1, (d+2)/(d+1))$ there exists a constant $\overline C$ depending only on $d, q$ and $\Omega$ such that any positive classical solution $\varphi$ of the backward heat equation
 \[
 \begin{cases}
 -\partial_t \varphi - \Delta \varphi  = 0 & \text{on $(0,t) \times \Omega$} \\
 \partial_n \varphi = 0 & \text{on $(0,t) \times \partial \Omega$} \\
 \int_\Omega \varphi(t,x) dx = 1,
 \end{cases}
 \]
satisfies
\[
\| \nabla \varphi \|_{L^q((0,t)\times \Omega)} \le \overline C (1+t)^{1/q}.
\]
We follow the strategy presented in \cite[Section 5]{GomesBook}. Note first that $\int_\Omega \varphi(s,x) dx = 1$ for all $s \in (0,t)$, by integrating by parts the equation and using the boundary conditions. We proceed in the case $d \ge 3$; if $d = 1$ or $d=2$, one argues in a similar way (see the discussion below). Let $\alpha  \in (0,1)$ to be chosen later; multiplying the equation by $\alpha \varphi^{\alpha-1}$ and integrating by parts yield for all $s \in (0,t)$
\[
\int_\Omega |\nabla \varphi^{\alpha/2}(s,x)|^2 dx = \frac{\alpha}{4(\alpha-1)}\partial_t \int_\Omega \varphi^\alpha(s,x) dx.
\]
Integrating in time and using the fact that $\int_\Omega \varphi(s,x) dx = 1$ give
\begin{equation}\label{est_grad}
\int_0^t \int_\Omega |\nabla \varphi^{\alpha/2}|^2 dx ds = \frac{\alpha}{4(1 - \alpha)} \int_\Omega \varphi^\alpha(0,x) dx - \frac{\alpha}{4(1 - \alpha)} \int_\Omega \varphi^\alpha(t,x) dx \le c_1,
\end{equation}
where $c_1$ depends on $d$ and $\Omega$ (the positive constants $c_2, c_3, \ldots$ used in the sequel will have the same dependance).

 We now exploit the continuous embedding of $W^{1,2}(\Omega)$ into $L^{\frac{2d}{d-2}}(\Omega)$; the adaption of this proof to the cases $d=1,2$ is straightforward, as the injection of $W^{1,2}(\Omega)$ is into $L^{p}(\Omega)$ for all $p \ge 1$. Hence, for all $s \in (0,t)$, by H\"older and Sobolev inequalities
 \begin{multline*}
 \int_\Omega \varphi^{\alpha + \frac{2}{d}}(s,x) dx \le \left(\int_\Omega \varphi(s,x) dx\right)^{\frac{2}{d}} \left(\int_\Omega \varphi^{\frac{\alpha}{2}\frac{2d}{d-2}}(s,x) dx\right)^{\frac{d-2}{d}} \\ \le c_2\left(\int_\Omega |\nabla \varphi^{\alpha/2}|^2 dx + \int_\Omega \varphi^\alpha dx \right)  \le  c_2\left(\int_\Omega |\nabla \varphi^{\alpha/2}|^2 dx + 1 +|\Omega| \right),
 \end{multline*}
 so
\begin{equation}\label{est_phi}
\int_0^t \int_\Omega \varphi^{\alpha + \frac{2}{d}} dx ds \le c_3 \left(\int_0^t \int_\Omega |\nabla \varphi^{\alpha/2}|^2 dx ds + t\right).
\end{equation}

Finally, since $q < (d+2)/(d+1)$, we may choose $\alpha \in (0,1)$ such that
\[
q\frac{2-\alpha}{2-q} = \alpha + \frac{2}{d},
\]
and therefore, by the identity $\nabla \varphi^{\alpha/2} = \frac{\alpha}{2} \varphi^{\frac{\alpha-2}{2}}\nabla \varphi$ and Young's inequality
 \begin{multline*}
\int_0^t \int_\Omega |\nabla \varphi|^q dx ds = \left(\frac{2}{\alpha}\right)^q \int_0^t \int_\Omega |\nabla \varphi^{\alpha/2}|^q \, \varphi^{q\frac{2-\alpha}{2}} dx ds \le \\ c_4 \left(\int_0^t \int_\Omega |\nabla \varphi^{\alpha/2}|^2 dx ds + \int_0^t \int_\Omega \varphi^{q\frac{2-\alpha}{2-q}} dx ds \right) \le c_4(c_1 + c_3(c_1 + t))),
\end{multline*}
in view of \eqref{est_grad} and \eqref{est_phi}, and the desired estimate follows.
 
 {\em Step 2.} Fix $t \in (0,T)$ and $1 < q < (d+2)/(d+1)$. Let $\varphi_0$ be any non-negative smooth function on $\Omega$ such that $\partial_n \varphi_0 = 0$ on $\partial \Omega$ and $ \int_\Omega \varphi_0(x) dx = 1$. Let $\varphi$ be the solution of the backward heat equation
  \[
 \begin{cases}
 -\partial_t \varphi - \Delta \varphi  = 0 & \text{on $(0,t) \times \Omega$} \\
 \partial_n \varphi = 0 & \text{on $(0,t) \times \partial \Omega$} \\
 \varphi(0,x) = \varphi_0(x) & \text{on $\Omega$.}
 \end{cases}
 \]
 Note that $\varphi$ is positive on $(0,t) \times \Omega$ by the strong maximum principle. Multiply the KFP equation in \eqref{mfg-syst}, integrate by parts and use the boundary conditions for $m_k$ to get
 \[
\int_0^t \int_\Omega  \partial_t m_k \, \varphi + \nabla m_k \cdot \nabla \varphi +D_p H_k(x,Dv_k) \cdot \nabla \varphi \, m_k \, dxds = 0.
 \]
 Integrating again by parts (in space-time) yields
  \[
\int_\Omega  m_k(t,x) \varphi_0(x) = \int_\Omega  m_k(0,x) \varphi(0,x) -  \int_0^t \int_\Omega D_p H_k(x,Dv_k) \cdot \nabla \varphi \, m_k \, dxds,
 \]
 using the equation and the boundary condition for $\varphi$. Hence,
 \begin{multline*}
\int_\Omega  m_k(t,x) \varphi_0(x) \le \| m_{k,0} \|_\infty + \|D_p H_k(\cdot,Dv_k)\|_{L^\infty((0,t) \times \Omega)} \int_0^t \int_\Omega  |\nabla \varphi| \, |m_k| \, dxds,\\
\le  \| m_{k,0} \|_\infty +  \overline C (1+t)^{1/q} \|D_p H_k(\cdot,Dv_k)\|_{L^\infty((0,t) \times \Omega)}  \|m_k\|_{L^{q'}((0,t) \times \Omega)} 
\end{multline*}
by Step 1. By the arbitrariness of $\varphi_0$, one obtains
\[
\| m_{k}(t, \cdot) \|_\infty \le \| m_{k,0} \|_\infty +  \overline C (1+t)^{1/q} \|D_p H_k(\cdot,Dv_k)\|_{L^\infty((0,t) \times \Omega)}  \|m_k\|_{L^{q'}((0,t) \times \Omega)} ,
\]
and since
\[
 \|m_k\|_{L^{q'}((0,t)\times\Omega)} \le \left( \int_0^t  \| m_{k}(s, \cdot) \|_\infty^{q'-1} \int_\Omega m_k(s,x) dx \, ds\right)^{1/q'} \le  \|m_k\|_{L^{\infty}((0,t)\times\Omega)}^{1/q}\, t^{1/q'},
\]
we have
\[
\| m_{k}(t, \cdot) \|_\infty \le \| m_{k,0} \|_\infty +  \overline C (1+t) \|D_p H_k(\cdot,Dv_k)\|_{L^\infty((0,t) \times \Omega)}  \|m_k\|_{L^{\infty}((0,t)\times\Omega)}^{1/q}.
\]
Passing to the supremum on $t \in (0,T)$, we conclude ($r$ in the statement can be chosen to be $q'$).

\end{proof}
   \begin{lem}[A mean-value theorem]
   \label{mvt}
    Let $\mathcal{K}\subseteq\R^d$, $f : \ol\Omega\times  \mathcal{K} \to \R^d$ be continuous and Lipschitz continuous in the second entry with constant $L$, uniformly in the first. Then there exists a measurable matrix-valued function $M(\cdot, \cdot,\cdot)$ such that 
  \begin{equation}
 \label{meanv}
   f(x,p) - f(x,q) = M(x,p,q)(p-q) , \quad |M(x,p,q)| \leq L 
    , \quad\forall \, x\in \ol\Omega, \,p,q\in \mathcal{K}.
 \end{equation}
   \end{lem}
 \begin{proof} Mollify $f$ in the variables $p$ and get a sequence $f_n$ converging to $f$ locally uniformly and with Jacobian matrix satisfying $\|D_pf_n\|_\infty\leq L$.
 Since $f_n$ is $C^1$ in $p$ the standard mean-value theorem gives
 \begin{equation}
 \label{meanv_n}
 f_n(x,p) - f_n(x,q) = \int_0^1 Df_n(x, q+s(p-q)) (p-q) \, ds =: M_n(x,p,q)(p-q) ,
 \end{equation}
 and $|M_n|\leq L$. We define the matrix $M$ componentwise by setting
 \[
 M(x,p,q)_{ij} := \liminf_n M_n(x,p,q)_{ij} , \quad i, j= 1,\dots, d,
 \]
 so that it is measurable in $(x,p,q)$ and satisfies $|M(x,p,q)| \leq L$. Now we take the $\liminf_n$ in the $i$-th component of the identity \eqref{meanv_n} and get the $i$-th component of the desired identity \eqref{meanv}.
\end{proof}
 \begin{proof} [Proof of Thm. \ref{st-uniq}.]  {\em Step 1.} First observe that, by the regularity of the solutions, $C_H<+\infty$ and $\bar C_H<+\infty$. We set
 \[ 
 v:=\tilde v - \ol v , \quad m:=\tilde m - \ol m, \quad
B_k(t,x):=\int_0^1D_pH_k(x, D\ol v(t,x)+s(D\tilde v-D\ol v)(t,x)) ds
\]
 and observe that $|B_k|\leq C_H$ 
 for all $k$ and $v_k$ satisfies
 \begin{equation}\label{Beq}
\left\{
\begin{array}{lll} -\partial_tv_k+B_k(t,x)\cdot Dv_k=\Delta v_k+F_k(x,\tilde m(t))-F_k(x,\ol m(t))
\quad 
 \text{ in  }(0, T)\times \Omega 
 \\ \\
\partial_n v_k = 0 \;  \text{ on }  (0,T)\times\partial \Omega , \quad  v_k(T,x)=G_k(x, \tilde m(T)) - G_k(x, \ol m(T)) .
  \end{array}
\right.\,
 \end{equation}

\noindent {\em Step 2.}
By the divergence theorem and the boundary conditions we compute
 \begin{multline*}
- \int_t^T\int_{\Omega}\partial_t v_k \Delta v_k \, ds =  \int_t^T\frac d{dt}\int_{\Omega}\frac{|Dv_k|^2}{2} dx ds - \int_t^T\int_{\de \Omega} \de_tv_k Dv_k\cdot n\, d\s \\ =
\frac 12 \|Dv_k(T,\cdot)\|_2^2 - \frac 12\|Dv_k(t,\cdot)\|_2^2 .
 \end{multline*}
Now we set
$$
\bar F(t,x):=F(x,\ol m)-F(x,\tilde m) , \quad  \bar G(t,x):=G(x,\ol m)-G(x,\tilde m) , 
$$
multiply the PDE in \eqref{Beq} by $\Delta v_k$, integrate,  use the terminal condition in  \eqref{Beq} and estimate
 \begin{multline*}
\frac 12\|Dv_k(t,\cdot)\|_2^2 + \int_t^T
\|\Delta v_k(s,\cdot)\|_2^2 \leq 
\frac 12 \|D\bar G(T,\cdot)\|_2^2 +
\\
\|B_k\|_\infty\int_t^T \left(\frac 1{2\eps}\|Dv_k(s,\cdot)\|_2^2 + \frac \eps 2 \|\Delta v_k(s,\cdot)\|_2^2\right) ds 
 + 
\int_t^T   \left(\frac 1{2\eps} \|\bar F(s,\cdot)\|_2^2
+ \frac \eps 2 \left\|\Delta v_k(s,\cdot)\right\|_2^2\right)  ds .
 \end{multline*}
Next we choose $\eps$ such that $1=(\|B_k\|_\infty+1)\eps/2$ and use 
the assumptions \eqref{Glip} and \eqref{Flip} to get
 \[
\|Dv_k(t,\cdot)\|_2^2 \leq 
L_{G} \|m
(T,\cdot)
\|_2^2 + \int_t^T  \frac {L_{F}}{\eps} \|m(s,\cdot)\|_2^2 ds +  \frac {\|B_k\|_\infty}{\eps}\int_t^T\|Dv_k(s,\cdot)\|_2^2 ds .
\]
Then Gronwall inequality gives, for $c_o:=
(\|B_k\|_\infty +1)/2=1
/\eps$ and for all $0\leq t\leq T$,
%
 \begin{equation}
 \label{est_1}
\|Dv_k(t,\cdot)\|_2^2 \leq \left(
L_{G} \|m(T,\cdot)\|_2^2 + c_o
L_F\int_t^T   \|m(s,\cdot)\|_2^2  ds\right)e^{ c_o
\|B_k\|_\infty T} .
 \end{equation}
 
 
\noindent {\em Step 3.} In order to write a PDE solved by $m$ we apply Lemma \ref{mvt} to $D_pH_k : \ol\Omega \times \mathcal{C} \to \R^d$, which is Lipschitz in $p$ by the 
 assumption in \eqref{ass:Hv2},  and get a matrix $M_k$ such that
\[
D_pH_k(x,D\ol v_k)-D_pH_k(x,D\tilde v_k) = M_k(x, D\ol v_k, D\tilde v_k)( D\ol v_k-D\tilde v_k),
\]
with $|M_k|\leq \bar C_H$. Now define 
\[
\tilde B_k(t,x):=D_pH_k(x,D\ol v_k) ,\quad A_k(t,x):= \tilde m_k
 M(x, D\ol v_k
 , D\tilde v_k
 ), \quad \tilde F_k(t,x):=A_k(t,x)(D\ol v_k-D\tilde v_k).
\]
Then  $m_k$ satisfies 
 \begin{equation}\label{m-linear}
\left\{
\begin{array}{lll}
\de_t m_k - \text{div}\left(
\tilde B_k 
m_k\right) = \Delta m_k +  \text{div}\tilde F_k 
  \quad 
\text{ in  }(0, T)\times \R^d, \\ \\
\de _n m_k + (m_k \tilde B_k +    \tilde F_k)\cdot n = 0   \;  \text{ on }  (0,T)\times\partial \Omega,\quad      m_k(0,x)= 0 .
 \end{array}
\right.\,
\end{equation}
with $|\tilde B_k|\leq C_H$ and $|A_k|\leq \mathcal{M} \bar C_H$ by the assumption \eqref{ass:Hv2}, where
\[
\mathcal{M}:=\max_k C [1+ \| m_{0,k}\|_\infty + (1+T)\|D_p H_k(\cdot,Dv_k) \|_{L^\infty((0,T)\times \Omega)}]^r
\]  
is the upper bound on $m_k$ given by  Proposition \ref{prop:est} (
where $C$ depends only on the set $\Omega$).

 \noindent {\em Step 4.}
We multiply the PDE in \eqref{m-linear} by $m_k$ and integrate by parts to get
 \begin{multline*}
0 =  \int_0^t\frac d{dt}\int_{\Omega}\frac{m_k^2}{2} \,dx ds +
\int_0^t\int_{\Omega}  |Dm_k|^2   \,dx ds - \int_0^t\int_{\de \Omega}  m_k Dm_k\cdot n \,d\s ds \\
 + \int_0^t\int_{\Omega}  m_k\tilde B_k\cdot Dm_k  \,dx ds - \int_0^t\int_{\de \Omega} m_k^2 \tilde B_k\cdot n \,d\s ds\\
 + \int_0^t\int_{\Omega}   \tilde F_k \cdot Dm_k  \,dx ds
- \int_0^t\int_{\de \Omega}  m_k \tilde F_k\cdot n  \,d\s ds.
 \end{multline*}
By the initial and boundary conditions in \eqref{m-linear} we obtain
 \begin{multline*}
\frac 12\|m_k(t,\cdot)\|_2^2 + \int_0^t\|Dm_k(s,\cdot)\|^2_2   \,ds = -\int_0^t\int_{\Omega}  \left(m_k\tilde B_k +  \tilde F_k\right)\cdot Dm_k  \,dx ds
\leq \\
\frac 1{2\eps} \int_0^t   \|\tilde F_k(s,\cdot)\|_2^2 ds +  \frac {\|\tilde B_k\|_\infty}{2\eps}\int_0^t\|m_k(s,\cdot)\|_2^2 ds + \eps\frac {\|\tilde B_k\|_\infty+1}{2} \int_0^t\|Dm_k(s,\cdot)\|^2_2   \,ds ,
 \end{multline*}
and with the choice $\eps=2/(\|\tilde B\|_\infty+1)=:1/c_1$
\[
 \|m_k(t,\cdot)\|_2^2 \leq  c_1
  \int_0^t   \|\tilde F_k(s,\cdot)\|_2^2 ds +  c_1
  \|\tilde B_k\|_\infty
  \int_0^t\|m(s,\cdot)\|_2^2 ds .
\]
Then Gronwall inequality and the definition of $\tilde F_k$ give, for all $0\leq t\leq T$,
 \begin{equation}
 \label{est_2}
\|m_k(t,\cdot)\|_2^2 \leq   
c_1 e^{c_1\|\tilde B_k\|_\infty T} \|A_k\|^2_\infty  \int_0^t   \| Dv_k(s,\cdot)\|_2^2 ds.
 \end{equation}

 \noindent {\em Step 5.} Now we set 
 $$\phi(t):= \|Dv(t,\cdot)\|_2^2=\sum_{k=1}^N\|Dv_k(t,\cdot)\|_2^2$$
  and 
 assume w.l.o.g. $C_H\geq1$, so that $c_o, c_1\leq C_H$.
 By combining \eqref{est_1} and \eqref{est_2} we get
  \begin{multline*}
\phi(t) \leq N e^{C_H^2 T} \left(
L_{G} \|m(T,\cdot)\|_2^2 + C_H
L_F\int_t^T   \|m(s,\cdot)\|_2^2  ds\right) \\
 \leq
\bar C_H ^2 C\left(L_G \int_0^T \phi(s) ds + C_H
L_F \int_t^T \int_0^\tau  \phi(s) ds \,d\tau \right) , \quad C:= N C_H e^{C_H^4 T^2} \mathcal{M}^2  .
 \end{multline*}
Then $\Phi := \sup_{0\leq t\leq T} \phi(t)$ satisfies
\[
\Phi\leq \Phi \Psi, \quad \Psi:= T\bar C_H^2 C(L_G+L_FC_HT/2),
\]
which implies $ \Phi=0$ if $\Psi < 1$. Therefore under such condition we conclude that $D\tilde v_k(t,x)=D\ol v_k(t,x)$ for all $k$, $x$ and $0\leq t\leq T$. By the uniqueness of solution for the KFP equation (e.g., Thm. I.2.2, p. 15 of \cite{LSU}) we deduce $\tilde m= \ol m$ and then, by the 
Comparison Principle for the HJB equation in the Appendix, $\tilde v= \ol v$.

 Finally, it is clear  that $\Psi$ can be made less than 1 by choosing either $T$, 
  or $\bar C_H$, or both $L_G$ and $L_F$ small enough.
 \end{proof} 
\subsection{Examples and remarks}

\begin{ex} 
\label{regul}
Integral costs. \upshape
Consider $F_k$ and $G_k$ of the form
\[
F_k(x, \mu)= F_o\left(x, \int_{\Omega} K(x,y) \mu(y) dy\right) , \quad G_k(x, \mu)= 
g_1(x)\int_{\Omega} \bar K(x,y) \cdot \mu(y) dy +g_2(x)
\]
with $F_o 
 : \ol\Omega
 \times \R^N\to \R$ measurable and Lipschitz in the second variable uniformly in the first, whereas   $K$ is an $N\times N$ matrix 
with components in   $L^2( \Omega\times\Omega)
$. Then 
 $F_k$ satisfies  \eqref{Flip}. 
 About $G_k$ we assume $g_1, g_2\in C^1(\Omega)$, $Dg_1$ bounded, the vector $\bar K$ and its Jacobian $D_x\bar K$ with components in $ L^2( \Omega \times \Omega)$. Then it
 satisfies
 \eqref{Glip}. Of course all the data $F_o, K, \bar K, g_i$ are allowed to change with the index $k=1,\dots, N$.
\end{ex} 
 \begin{ex} Local costs. 
 \label{local}  \upshape
 Take $G_k=G_k(x)$ independent of $m(T)$ and $F_k$ of the form
$ F_k(x,\mu)=F_k^l(x,\mu(x))
 $ 
 with $F_k^l: \ol\Omega
 \times [0,+\infty)^N\to\R$ measurable and Lipschitz in the second variable uniformly in the first.
  Then $F_k$ satisfies \eqref{Flip}. 
\end{ex}
\begin{ex}
Costs depending on the moments. \upshape
The mean value of the density $\mu$, $M(\mu)=\int_{\Omega} y\mu(y) dy$,  and all its moments 
$\int_{\Omega} y^j\mu(y) dy$, $j= 2, 3, \dots$,
are Lipschitz in $L^2$ 
by Example \ref{regul}. 
Then any $F_k$ (resp., $G_k$) depending on $\mu$ only via these quantities 
 satisfies \eqref{Flip} (resp., \eqref{Glip}) if it is Lipschitz with respect to them uniformly with respect to $x$.
\end{ex}
 \begin{ex} Convex Hamiltonians. 
 \label{Bell}  \upshape The usual Hamiltonians in MFGs are those arising from classical Calculus of Variations, e.g., $H _k(x,p)= b_k(x)(c_k + |p|^2)^{\b_k/2}$, which satisfies the assumption \eqref{ass:H} if $b_k\in C(\ol\Omega)$ and either $c_k>0$ or $c_k=0$ and $\b_k\geq 2$. 
 
 A related class of Hamiltonians are those of Bellman type associated to nonlinear systems, affine in the control $\a\in \R^d$,
 \begin{equation}
 \label{bellman}
 H _k(x,p):=\sup_\a\{ -(f_k(x)+g_k(x)\a)\cdot p -L_k(x,\a)\} = -f_k(x)\cdot p + L_k^*\left(x, -g_k(x)^Tp\right) ,
 \end{equation}
 where $f_k$ is a Lipschitz vector field, $g_k$ a Lipschitz square matrix, $L_k(x,\a)$ is the running cost of using the control $\a$ (adding to $F_k(x,m)$ in the cost functional of a representative player), and $L_k^*(x, \cdot)$ is its convex conjugate with respect to $\a$. In this case one can check the assumption \eqref{ass:H} on an explicit expression of $L_k^*$. For instance, if $L_k(x,\a)=|\a|^\g/\g$ then
 \[
 H _k(x,p)= -f_k(x)\cdot p + \frac{\g -1}{\g}|g_k(x)^Tp|^{\g/(\g -1)} ,
 \]
 which satisfies \eqref{ass:H} if $\g\leq 2$.
 \end{ex}
 \begin{ex} Nonconvex Hamiltonians. 
 \label{Isa} \upshape
Two-person 0-sum differential games give rise to the Isaacs Hamiltonians, which are defined in a way similar to \eqref{bellman} but as the  inf-sup over two sets of controls. A motivation for considering these Hamiltonians in MFGs is proposed in \cite{T}.  
A relevant
 example 
  is the case of robust control, or nonlinear $H_\infty$ control, studied in connection with MFGs by \cite{BBT, moonbasar16} (see also the references therein). In this class of problems a deterministic disturbance $\s(x) \b$ affects the control system ($\s$ is a Lipschitz square matrix) and a worst case analysis is performed by assuming that $\b\in \R^d$ is the control of an adversary who wishes to maximise the cost functional of the representative agent; a term $-\d |\b|^2/2$, with $\d>0$ is added to the running cost to penalise the energy of the disturbance. The Hamiltonian for robust control then becomes
 \begin{equation}
 \label{isaacs}
 H_k^{(r)}(x,p) := H _k(x,p) + \inf_{\beta}\left\{-\s(x) \b + \d |\b|^2/2\right\} \cdot p = H _k(x,p) - \frac{|\s(x)^Tp|^2}{2\d} ,
  \end{equation}
  which is the sum of the convex $H _k$ of the previous example and a concave function of $p$. Clearly it satisfies the condition 
   \eqref{ass:H} if and only if $H_k$ does. 
  \end{ex}
 \begin{rem} \label{contdip}
 Continuous dependence on data. \upshape
Our proof of uniqueness can be adapted to show the Lipschitz dependence of solutions on some data. For instance, in Theorem \ref{st-uniq} we may assume that $\bar m(0,x)=\bar m_0(x)$ and $\tilde m(0,x)=\tilde m_0(x)$, with $\bar m_0, \tilde m_0 \in { \mathcal{P}_N}(\Omega)$. Then a simple variant of the proof allows to estimate 
$$
\|\tilde m(t,\cdot) - \bar m(t,\cdot)\|_2^2 \leq \frac C\delta \|\tilde m_0 - \bar m_0\|_2^2
$$
where $0<\d \leq 1-\Psi$ and $C$ depends on the same quantities as $\Psi$. A similar estimate holds for  $\|D\tilde v(t,\cdot) - D\bar v(t,\cdot)\|_2^2$. Under some further assumptions on the costs $F$ and $G$ one can also use results on the HJB equation to obtain the continuous dependence of $v$ itself upon the initial data $m_0$. More precise results on continuous dependence of solutions with respect to data will be given elsewhere.
  \end{rem}
\begin{rem}
\label{plions}
\upshape
The statement of Theorem \ref{st-uniq} holds with the same  proof for solutions $\Z^d$-periodic in the space variable $x$ in the case that $F_k$ and $G_k$ are $\Z^d$-periodic in $x$ and without Neumann boundary conditions. 
In such case of periodic boundary conditions a uniqueness result for short $T$ was presented by Lions in \cite{L} for $N=1$, regularizing running cost $F$, and for terminal cost $G$ independent of $m(T)$. He used estimates in $L^1$ norm for $m$ and in $L^\infty$ norm for $Dv$, instead of the $L^2$ norms we used here in \eqref{est_1} and \eqref{est_2}. See also \cite {BF} for the case of a single population.
\end{rem}
\begin{rem}\upshape
The constants $C_H$ and $\bar C_H$ in the theorem depend only on the data of \eqref{mfg-syst} if $H_k$ and $D_pH_k$ are globally  Lipschitz in $p$, uniformly in $x$, for all $k$.
In this case the smallness condition $\Psi <1$ 
does not depend on the solutions $\tilde v, \ol v$. In the next section we reach the same conclusion for much more general Hamiltonians $H_k$ under some mild additional conditions on the costs $F_k, G_k$.
\end{rem}
\begin{rem}\upshape
If the volatility is different among the populations the terms $\Delta v_k, \Delta m_k$ in \eqref {mfg-syst} are replaced, respectively, by $\nu_k\Delta v_k$ and $\nu_k \Delta m_k$. If the constants $\nu_k$ are all positive, the theorem remains true with the function $\Psi$ now depending also on $\nu_1,\dots, \nu_N$ and minor changes in the proof. The case of volatility depending on $x$ leads to  operators of the form  $\text{trace} (\sigma_k(x)\sigma^T_k(x)D^2v_k)$ in the the HJB equations and their adjoints in the KFP equations. This can also be treated, with some additional work in the proof, if such operators are uniformly elliptic, i.e., the minimal eigenvalue of the matrix $\sigma_k(x)\sigma^T_k(x)$ is bounded away from 0 for $x\in\ol \Omega$.
\end{rem}
\begin{rem}\upshape
The $C^{2,\beta}$ regularity of $\de\Omega$ can be weakened in Theorem \ref{st-uniq}. Here we used, e.g., Theorem IV.5.3 of \cite{LSU} to produce a smooth test function $\varphi$ in the proof of Proposition \ref{prop:est}. However, we could work instead with a weak solution of the backward heat equation, which exists,  for instance, if $\de\Omega\in C^{1,\beta}$ by Theorem 6.49 of \cite{Lieb}, or if it is "piecewise smooth" by Theorem III.5.1 in \cite{LSU}.
\end{rem}
%
%
\section{Special cases and applications}\label{appl}
The function $\Psi$ of Theorem \ref{st-uniq} may depend on the solutions $\tilde v, \ol v$ if the Hamiltonians $H_k$ are not globally Lipschitz or they have unbounded second derivatives, because the constants $C_H, \bar C_H$ may depend on the range of $D\tilde v$ and $D \ol v$. Under some further assumptions we can estimates these quantities and therefore get a uniqueness result where the function $\Psi$ depends only on the data of the problem \eqref{mfg-syst}. The additional assumptions are
\begin{equation}
\label{FGbdd}
|F_k(x,\mu)|
 \leq C_F , \quad |G_k(x,\mu)| \leq C_G , \quad \forall \, x\in\ol\Omega, \mu\in{ \mathcal{P}_N}(\Omega) , k=1,\dots,N,
\end{equation}
\begin{equation}
\label{growth}
|H_k(x,p)| \leq \a (1+|p|^2) , \quad |D_pH_k(x,p)|(1+|p|)\leq \a (1+|p|^2) ,\quad \forall \, x
, p
, k,
\end{equation}
and $x\to G(x,\mu)$ of class $C^2$, for all $\mu\in\mathcal{P}_N(\Omega)$, with
\begin{equation}
\label{G'bdd} 
\|DG_k(\cdot,\mu)\|_\infty + \|D^2 G_k(\cdot,\mu)\|_\infty \leq  C'_G , \forall \, k .
\end{equation}
\begin{cor}
\label{cor1}
Assume \eqref{ass:H}, 
 \eqref{Flip}, \eqref{Glip}, $m_0\in { \mathcal{P}_N}(\Omega)$, \eqref{FGbdd},  \eqref{growth}, and  \eqref{G'bdd}.
\noindent
Then there exists 
$\ol T>0$
such that for all $T\in (0, \bar T]$ can be at most one 
classical solution of \eqref{mfg-syst}.
\end{cor}
\begin{proof}
By Assumption \eqref{FGbdd} the functions $\pm (C_G+t(C_F+\alpha))$ are, respectively, a super- and a subsolution of the HJB equation in \eqref{mfg-syst} with homogeneous Neumann condition and terminal condition $G_k$, for any $k$ and $m$. Then the Comparison Principle in the Appendix gives for any solution of \eqref{mfg-syst} the estimate
\[
|v_k(t,x)| \leq C_G+TC_F ,\quad \forall\, (t,x)\in [0,T]\times \ol\Omega , k=1,\dots,N .
\]
Now we can use an estimate of Theorem V.7.2, p. 486 of \cite{LSU}, stating that there is a constant $K$, depending only on $\max|v_k|, \a, C'_G$, and $\de \Omega$, such that
\[
|Dv_k(t,x)| \leq K ,\quad \forall\, (t,x)\in [0,T]\times \ol\Omega , k=1,\dots,N .
\]
Then the constant $C_H$ in 
\eqref{ass:Hv1} is bounded by 
$C'_H:=\a(1+K^2)/(1+K)$, and $\bar  C_H$ defined by \eqref{ass:Hv2} can be estimated by 
\[
 \bar C'_H := \max_{k= 1,\dots, N} \sup_{x\in\Omega,\, |p|, |q|\leq K} \frac{|D_pH_k(x,p)-D_pH_k(x,q)|}{|p-q|} .
 \]
Now Theorem \ref{st-uniq} gives the conclusion.
\end{proof}
\begin{rem}\upshape
The constant $\bar T$ in the Corollary depends only on $L_F, L_G, N, \a, C_F, C_G,$ $C'_G$,  $\max_k\| m_{0,k}\|_\infty$, $\Omega$ and the constants $C'_H, \bar C'_H$ built in the proof. A similar results holds if, instead of $T$ small, we assume $L_F$ and $L_G$ suitably small.
\end{rem}
\begin{ex} 
\label{regul2}
Costs satisfying the assumptions. \upshape
The nonlocal cost $F_k$ and $G_k$ of Example \ref{regul} satisfy Assumption \eqref{FGbdd} if, for instance,  
$K, \bar K$, and $g_i$ are 
bounded and $F_o$ is continuous.

The Assumption \eqref{G'bdd}  is verified if $g_1, g_2\in C^2(\ol\Omega)$ and $|D^2_x\bar K(x,y)|+|D^2_x\bar K(x,y)|\leq C$ for all $x, y$.

For the local cost $F_k$ of Example \ref{local},  \eqref{FGbdd} holds if $F^l_k$ is bounded.
 \end{ex}

 \subsection{Well-posedness of segregation models}
 
 Next we combine this uniqueness result with an existence theorem for models of urban settlements and residential choice proposed in \cite{ABC}. We take for simplicity
 \begin{equation}
\label{simple} 
N=2 , \quad G_k\equiv 0 ,\quad H_k(x,p)=h_k(x, |p|) .
\end{equation}
 We endow $\mathcal{P}_2(\Omega)$ with the Kantorovitch-Rubinstein distance and strengthen condition \eqref{FGbdd} to
  \begin{equation}
\label{strength} 
(F_1, F_2) : \overline\Omega\times \mathcal{P}_2(\Omega)\to \R^2 \; \text{ continuous and with bounded range in } C^{1,\b}(\ol\Omega),
\end{equation}
 for some $\b>0$. We also assume a compatibility condition and further regularity on $m_0$:
   \begin{equation}
\label{strength2} 
\de _n m_{0,k} = 0 \; \text{ on } \de\Omega , \quad m_{0,k}\in C^{2,\b}(\ol\Omega) , \quad k=1, 2 .
\end{equation}
 \begin{cor}
 \label{cor2}
Assume \eqref{ass:H}, 
 \eqref{Flip},  \eqref{growth}, \eqref{simple},  \eqref{strength},  \eqref{strength2}, and $H_k\in C^1(\ol\Omega\times \R^d)$. 
\noindent
Then there exists $\ol T>0$
such that for all $T\in (0, \bar T]$ there exists 
a unique classical solution of \eqref{mfg-syst}.
\end{cor}
\begin{proof} The existence of a solution (for any $T$) follows from Theorem 12 of \cite{ABC}. Let us only note that, by \eqref{simple}, $D_pH_k(x,p)=\de_{|p|} h_k(x, |p|) p/|p|$, and then the compatibility condition in \eqref{strength2} and the Neumann condition for $v_k$ imply also the compatibility condition
   \begin{equation}
   \label{comp}
  \partial_n m_{0,k} + m_{0,k} D_p H_k(x, Dv_k(0,x)) \cdot n = 0  \quad \forall \, x\in\partial \Omega .
\end{equation}
The uniqueness of the solution for small $T$ follows from Corollary \ref{cor1}.
 \end{proof}
 \begin{rem}\upshape
Here the constant $\bar T$ depends on $L_F, \a, C_F,$ $\max_k\| m_{0,k}\|_\infty$, $\Omega$, and the constants $C'_H, \bar C'_H$ built in the proof of Corollary \ref{cor1}. 
The solution $m$ and $Dv$ depend in a Lipschitz way from the initial condition $m_0$, as explained in Remark \ref{contdip}.
\end{rem}
\begin{ex} 
\label{Sche} 
Costs of Schelling type. \upshape
Let $K_k : \ol\Omega \times \ol \Omega \to \R$ be Lipschitz and such that, for some 
$U(x)$  neighborhood of $x$, $K_k(x,y)=1$ for $y\in U(x)$ and $K_k(x,y)=0$ for $y$ out of a small neighborhood of $U(x)$. Then
\[
N_k(x,\mu_k
) := 
{\int_{\Omega} K_k(x,y) \mu_k(y) dy}
\]
represents the amount of population $k$ around $x$. The cost functional for the $k$-th population introduced in \cite{ABC} and inspired by the studies on segregation of T. Schelling \cite{Sch} is of the form 
\[
F_k(x,\mu_1,\mu_2
) :=\left(\frac{N_k(x,\mu_k)}{N_k(x,\mu_k) + N_{3-k}(x,\mu_{3-k}) + \eta} - a_k\right)^- ,
\]
where $(\;)^-$ denotes the negative part and $\eta>0$ is very small. It means that if the ratio of the $k$-th population with respect to the total population in the neighborhood of $x$ is above the threshold $a_k$, then a representative agent of this population is happy because his cost is 0, whereas below the threshold the agent incurs in a cost and therefore he wants to move from the neighborhood.  These costs fall within Example \ref{regul2} and satisfy \eqref{Flip} and \eqref{FGbdd}. Moreover $F_k : \Omega \times  \mathcal{P}_2(\Omega)\to \R$ is Lipschitz.

To meet the assumptions of Corollary \ref{cor2} we assume the kernel $K$ is of class $C^2$ in $x$ and we approximate the negative part $(\;)^-$ with a smooth function, e.g.,
\[
\varphi_\eps (r) := \frac{\sqrt{r^2 + \eps^2} - r}2 , 
\]
for a small $\eps>0$. Then the cost functionals
\[
F_k^\eps(x,\mu_1,\mu_2
) :=\varphi_\eps \left(\frac{N_k(x,\mu_k)}{N_k(x,\mu_k) + N_{3-k}(x,\mu_{3-k}) + \eta} - a_k\right) 
\]
satisfy also \eqref{strength}.
\end{ex}
\begin{ex} Hamiltonians.
 \upshape Typical examples are either $H _k(x,p)= b_k(x)|p|^2$, with $b_k\in C(\ol\Omega)$, or
 \[
H _k(x,p)=b_k(x)(1+|p|^2)^{\b_k/2} ,\quad 0<\b_k\leq 2 .
 \]
They satisfy \eqref{ass:H} and \eqref{growth}, moreover they are in $C^1(\ol\Omega\times \R^d)$ if $b_k\in C^1(\ol\Omega)$.
\end{ex}
%
 \begin{rem}
 \upshape
 In the last Corollary \ref{cor2} the simplifying assumption $G_k\equiv 0$ can be dropped and replaced with $G_k : \overline\Omega\times \mathcal{P}_2(\Omega)\to \R$ continuous, with bounded range in  $C^{2,\b}(\ol\Omega)$, and satisfying \eqref{Glip}. Then \eqref{G'bdd} holds and the constant $\bar T$ depends also on $ L_G, C_G,$ and $C'_G$. Examples of such terminal costs can be given along the lines of Examples \ref{regul}, \ref{regul2}, and \ref{Sche}.
 \end{rem}

 \subsection{Well-posedness of robust Mean Field Games}  
 For simplicity we limit ourselves to a single population of agents, so $N=1$ and we drop the subscripts $k$. The representative agent has the dynamics in $\R^d$
 \[
 dX_s = \left(f(X_s)+g(X_s)\a_s + \s(X_s) \b_s\right) ds + dW_s ,
 \]
 where $f$ is a $C^1$ 
  vector field in $\ol\Omega$, $g$ and $\s$ are  $C^1$ 
   scalar functions  in $\ol\Omega$, $W_s$ is a $d$-dimensional Brownian motion, $\a_s, \b_s$ take values in $\R^d$ and are, respectively, the control of the agent and a disturbance affecting the system. The cost functional is (for $\d>0$)
 \[
\mathbb E\left[ \int_0^T \left(F(X_s, m(s, \cdot)) + \frac{|\a_s|^2}2  -\d \frac{|\b_s|^2}2  \right) \, ds +G(X_T, m(T,\cdot)) \right]
 \]
 that the agent wants to minimise whereas the disturbance, modeled as a second player in a 2-person 0-sum game, wants to maximise. This leads to the Hamiltonian
\begin{equation}\label{Hrob}
 H (x,p)= -f(x)\cdot p + g^2(x)\frac{|p|^2}2  -   \s^2(x)\frac{|p|^2}{2\d} .
\end{equation}
Note that here $g(x)$ and $\sigma(x)$ are scalars, different 
 from Examples \ref{Bell} and \ref{Isa}.
  On the costs we assume
    \begin{equation}
\label{strength3} 
F, \, G : \overline\Omega\times \mathcal{P}_1(\Omega)\to \R \; \text{ continuous with bounded range, resp., in } C^{1,\b}(\ol\Omega) \text{ and  } C^{2,\b}(\ol\Omega)
\end{equation}
 for some $\b>0$. The compatibility condition and 
 regularity on $m_0$ now are
   \begin{equation}
\label{strength4} 
\de _n m_{0} - m_0 f\cdot n= 0 \; \text{ on } \de\Omega , \qquad m_{0}\in C^{2,\b}(\ol\Omega)  .
\end{equation}
 \begin{cor}
 \label{cor3}
Assume $N=1$ with the Hamiltonian defined by  \eqref{Hrob}, 
 \eqref{Flip},  \eqref{Glip},  \eqref{strength3},  and \eqref{strength4}. 
\noindent
Then for all $T>0$ there is 
a  classical solution of \eqref{mfg-syst}, and there exists $\ol T>0$
such that for all $T\in (0, \bar T]$ such solution is unique.
\end{cor}
\begin{proof} The existence of a solution 
 follows from Theorem 12 of \cite{ABC}. In fact, $H\in C^1(\ol\Omega\times\R^d)$ and it has quadratic growth. Moreover 
 $$
 D_pH(x,p)=-f(x) + g^2(x)p - \frac{\s^2(x)}{\d}p ,
 $$
  and then the compatibility condition in \eqref{strength4} and the Neumann condition for $v$ imply again the compatibility condition \eqref{comp}.

The uniqueness of the solution for small $T$ follows from Corollary \ref{cor1}, since $H$  satisfies also \eqref{ass:H}.
 \end{proof}
 \begin{rem}\upshape
Also here the solution $m$ and $Dv$ depend in a Lipschitz way from the initial condition $m_0$, as explained in Remark \ref{contdip}.
\end{rem}
 \begin{rem}  \upshape
Our example of robust MFG is different from the one in \cite{BBT}. In that paper the state space is $\Omega=\R$, one-dimensional  without boundary, the control system is linear in the state $X_s$, and the volatility is $\s X_s$ instead of 1, for some positive constant $\s$, so the parabolic operators in the HJB and KFP equations of \eqref{mfg-syst} are degenerate at the origin. The well-posedness of the MFG system of PDEs in \cite{BBT} is an open problem.
  \end{rem}

\section{Appendix: a Comparison Principle}\label{app}
The next result is known but we give its elementary proof for lack of a precise reference. 
\begin{prop}
Assume $\Omega\subseteq\R^d$ is bounded with $C^2$ boundary, $H : \ol\Omega \times \R^d$ is of class $C^1$ with respect to $p$,  and $u, v : [0,T]\times \ol\Omega \to \R$ are $C^1$ in $t$ and $C^2$ in $x$ and satisfy
\begin{equation*}
\left\{
\begin{array}{ll}
- \partial_t u -  \Delta u + H(x,Du) \leq - \partial_t v -  \Delta v + H(x,Dv)  , & \textit{in }  (0,T)\times\Omega , \\ \\
\partial_n u \leq \partial_n v ,   & \textit{on }  (0,T)\times\partial \Omega , \\ \\
u(T,x) \leq v(T,x)  & \textit{in } \Omega.
\end{array}
\right.
\end{equation*}
Then $u \leq v$ in $[0,T]\times\ol\Omega$.
\end{prop}
\begin{proof}
Let us assume first that 
\[
- \partial_t (u-v) -  \Delta (u-v) + H(x,Du) - H(x, Dv) < 0 \quad\text{ in } [0,T)\times\Omega ,
\]
 $\partial_n (u-v) <0\;$  {on }  $[0,T)\times\partial \Omega$, and $(u-v)(T,x) \leq \d$. Then the maximum of $u-v$ can be attained only at $t=T$, which implies $u-v
 \leq \d$ in  $[0,T]\times\ol\Omega$.
 
 Now take $g\in C^2(\ol\Omega)$ such that $Dg(x)=n(x)$ for all $x\in\de\Omega$ and define
 \[
 v_\eps(t,x):= v(t,x)+\eps(T-t)C+\eps g(x).
 \]
 Then $\;\partial_n (u-v_\eps)=\partial_n (u-v) -\eps <0\;$ and $\;(u-v_\eps)(T,x) \leq \eps\|g\|_\infty$. Moreover, by Taylor's formula, for some $q$ with $|q|\leq \|Dg\|_\infty$, 
 \begin{multline*}
 - \partial_t (u-v_\eps) -  \Delta (u-v_\eps) + H(x,Du) - H(x, Dv_\eps) = \\
 - \partial_t (u-v) -  \Delta (u-v) + H(x,Du) - H(x, Dv)  -\eps (C -\Delta g+D_pH(x,q)\cdot Dg)<-\eps
 \end{multline*}
 if $C$ is chosen large enough. Then
 \[
 u\leq v_\eps +  \eps\|g\|_\infty \leq v +\eps(TC+ 2\|g\|_\infty) 
 \]
 and we conclude by letting $\eps \to 0$.
\end{proof}
\begin{rem}\upshape
The result remains true if $\de \Omega$ is merely $C^1$ and satisfies an interior sphere condition. 
This can be proved in a less direct way by 
linearizing the inequality for $u-v$ and then using 
 the parabolic Strong Maximum Principle and the parabolic version of Hopf's Lemma for linear equations (see, e.g., \cite{PW}).
\end{rem}

\end{document}